\documentclass{amsart}
\usepackage{epsfig, amssymb, amsfonts}
\usepackage[all]{xy}
\usepackage[margin=1in]{geometry}

\newtheorem{lemma}{Lemma}[section]
\newtheorem{prop}[lemma]{Proposition}
\newtheorem{thm}[lemma]{Theorem}

\newtheorem{cor}[lemma]{Corollary}

\newtheorem{lem}[lemma]{Lemma}
\newtheorem{rem}[lemma]{Remark}

\newtheorem*{special theorem}{My Specially-Named Theorem}

\newcommand{\Z} { {\mathbb Z} }
\newcommand{\Q} { {\mathbb Q} }
\newcommand{\Spec} { {\mathrm{Spec~}} }
\newcommand{\F} { {\mathbb F} }

\newcommand{\colim} { {\mathrm{colim}} }

\newcommand{\comment}[1]{}

\title{On some negative motivic homology groups}
\date{}
\author{Tohru Kohrita}
\address{Graduate School of Mathematics, Nagoya University, Nagoya, Japan}
\email{kohrita.tohru@j.mbox.nagoya-u.ac.jp}

\begin{document}
\maketitle

\begin{abstract}
For an arbitrary separated scheme $X$ of finite type over a finite field $\F_q$ and a negative integer $j,$ we prove under the assumption of resolution of singularities, that $H_{-1}(X,\Z(j))$ is canonically isomorphic to $H_{-1}(\pi_0(X),\Z(j))$ if $j=-1$ or $-2,$ and $H_i(X,\Z(j))$ vanishes if $i\leq-2$ and $i-j\leq1.$ As the group $H_{-1}(\pi_0(X),\Z(j))$ is explicitly known, this gives a explicit calculation of motivic homology of degree $-1$ and weight $-1$ or $-2$ of an arbitrary scheme over a finite field.
\end{abstract}

\section{Introduction}

In this paper, we assume that schemes are separated and of finite type over a perfect field. The finite field with $q$ elements is written as $\F_q.$ For a scheme $X,$ $\pi_0(X)$ denotes the spectrum of $\mathcal O_X(X)^{\acute et},$ the largest \'etale $k$-algebra contained in $\mathcal O_X(X)$ that is finite over $k.$ The properties of $\pi_0(X)$ relevant to us can be found in \cite[p.495-496]{Liu}.

The aim of this paper is to prove the following theorem on motivic homology.

\begin{thm}\label{thm:MAIN}
Assume that resolution of singularities holds over $\F_q.$ Let $i$ and $j$ be negative integers. Then, for all schemes $X$ over $\F_q,$ $$H_i(X,\Z(j))=0$$ if $i\leq-2$ and $i-j\leq1.$ In degree $i=-1,$ the canonical map
$$\alpha_X:H_{-1}(X,\Z(j))\longrightarrow H_{-1}(\pi_0(X),\Z(j))$$
is an isomorphism if $i-j\leq 1,$ i.e. $j=-1$ or $-2.$
\end{thm}

Since $\pi_0(X)$ is finite \'etale over $\F_q,$ it is a finite disjoint union of spectra of finite fields. Hence, the isomorphism $H_{-1}(X,\Z(j))\cong H_{-1}(\pi_0(X),\Z(j))$ of Theorem~\ref{thm:MAIN} and the explicit computation of negative motivic homology groups of finite fields as in Lemma~\ref{rem:calculation} give an explicit computation of $H_{-1}(X,\Z(j))$ ($j=-1$ or $-2$) for an arbitrary scheme $X$ over $\F_q.$ In particular, if $X$ is geometrically connected over $\F_q$ (this is equivalent to requiring that $X$ be connected and $\pi_0(X)=\Spec\F_q$ \cite[Chapter 10, Corollary 2.21(a)]{Liu}), we have the following corollary. 

\begin{cor}
Under resolution of singularities, if $X$ is a geometrically connected scheme over $\F_q$ and $j=-1$ or $-2,$ there is a canonical isomorphism
$$H_{-1}(X,\Z(j))\buildrel\sim\over\longrightarrow H_{-1}(\Spec \F_q,\Z(j))\cong \F_{q^{-j}}^\times.$$ 
\end{cor}

\begin{rem}
It is worth noting that if one assumes Parshin's conjecture, the statement in Theorem~\ref{thm:MAIN} holds for all negative integers $i$ and $j$ without the bound $i-j\leq 1.$ One only needs to invoke \cite[Proposition 4.1]{KY} instead of Proposition~\ref{thm:KYParshin} in order to prove the claims of Lemma~\ref{lem:4.1}, Proposition~\ref{prop:smcpt} and Proposition~\ref{prop:resolution} without the bounding conditions on $i-j.$ One may similarly prove Proposition~\ref{lem:smoothisom}, Lemma~\ref{lem:desingisom} and Lemma~\ref{lem:surj} for all negative integers $j.$
\end{rem}

Theorem~\ref{thm:MAIN} is a version in the context of motivic homology of the following theorem of Kondo and Yasuda on Borel-Moore motivic homology. In fact, if the scheme $X$ is proper, our Theorem~\ref{thm:MAIN} is due to Kondo and Yasuda:
 
\begin{thm}[{\cite[Theorem 1.1]{KY}}]\label{thm:KYmain}
Let $j=-1$ or $-2$ and let $X$ be a connected scheme over a finite field $\F_q.$ If $X$ is not proper,
$$H_{-1}^{BM}(X,\Z(j))=0.$$
If $X$ is proper, the pushforward maps 
$$H_{-1}^{BM}(X,\Z(j))\longrightarrow H_{-1}^{BM}(\Spec\mathcal O_X(X),\Z(j))$$
are isomorphisms. 
\end{thm}

Theorem~\ref{thm:KYmain} itself is a generalization of an earlier result of Akhtar \cite[Proposition 3.1]{Ak}, where the claim for $j=-1$ was proved for smooth projective schemes $X.$ 

The case $i\leq-2$ of our Theorem~\ref{thm:MAIN} is also due to Kondo and Yasuda if the scheme $X$ is proper (see Proposition~\ref{thm:KYParshin}).

Theorem~\ref{thm:MAIN} and Theorem~\ref{thm:KYmain} are related as follows. If $X$ is a scheme over $\F_q,$ $\Spec\mathcal O_X(X)$ is also a scheme over $\F_q.$ (Recall our convention on schemes.) Thus, the canonical factorization
$$X\longrightarrow \Spec\mathcal O_X(X)\longrightarrow \pi_0(X)\longrightarrow\Spec\F_q$$
of the structure morphism of $X$ gives, on applying Theorem~\ref{thm:MAIN}, the isomorphisms
$$H_i(X,\Z(j))\buildrel\sim\over\longrightarrow H_i(\Spec\mathcal O_X(X),\Z(j))\buildrel\sim\over\longrightarrow H_i(\pi_0(X),\Z(j))$$
because $\pi_0(\Spec\mathcal O_X(X))=\pi_0(X)$ by definition. It is trivially true that the first map is an isomorphism if $X$ is affine and so is the second if $X$ is proper. Since motivic homology and Borel-Moore homology agree for proper schemes, the theorem of Kondo and Yasuda says that the first map is an isomorphism when $X$ is proper. They proved this without assuming resolution of singularities. Our Theorem~\ref{thm:MAIN} claims that both maps are always isomorphisms if we assume the existence of resolution of singularities.

Let us end this introduction with a summary of the properties of motivic homology and cohomology theories which we shall use freely in the subsequent sections. 

By motivic (co)homology (with compact supports) with coefficients in an abelian group $A,$ we mean the following four theories defined for schemes $X$ over a perfect field $k$:

\begin{itemize}
\item motivic homology, 
$$H_i(X,A(j)):=Hom_{DM_{Nis}^{-}(k)}(A(j)[i],M(X));$$
\item motivic cohomology, 
$$H^i(X,A(j)):=Hom_{DM_{Nis}^{-}(k)}(M(X),A(j)[i]); $$
\item motivic homology with compact supports,
$$H_i^{BM}(X,A(j)):=CH_j(X,i-2j;A);$$ 
\item motivic cohomology with compact supports, 
$$H_c^i(X,A(j)):=Hom_{DM_{Nis}^{-}(k)}(M^c(X),A(j)[i]).$$
\end{itemize}
Here $DM_{Nis}^-(k)$ is Voevodsky's triangulated category of motives \cite[Chapter 5]{VSF} and $CH_j(X,i-2j;A)$ is Bloch's higher Chow group \cite{Bloch}. We refer to motivic homology with compact supports as Borel-Moore homology. We index higher Chow groups ``homologically" by dimension of cycles contrary to the more common indexing by codimension of cycles. With this indexing $CH_r(X,s;A)$ is a subquotient of the group of cycles of dimension $r+s$ in $X\times\Delta^s$ that intersect properly with all faces. The advantage of this convention is that we do not need to require $X$ to be equi-dimensional. If $X$ is equi-dimensional, we have $CH_r(X,s;A)=CH^{\dim X-r}(X,s;A).$

There is a canonical isomorphism (\cite[Chapter 5, Proposition 4.2.9]{VSF}, \cite[Proposition 19.18]{MVW})
$$Hom_{DM_{Nis}^{-}(k)}(\Z(j)[i],M^c(X))\buildrel\sim\over\longrightarrow CH_j(X,i-2j)$$
if $X$ is quasi-projective and $k$ admits resolution of singularities in the sense of \cite[Chapter 4, Definition 3.4]{VSF} . We chose to define Borel-Moore homology by Bloch's higher Chow groups mainly because they have localization sequences without assuming resolution of singularities. For motivic cohomology, it is known that if $X$ is smooth and of pure dimension $d$, there is a canonical isomorphism \cite[Theorem 19.1]{MVW}:
$$H^i(X,A(j))\buildrel\sim\over\longrightarrow CH_{d-j}(X,2j-i;A)\buildrel def\over =H_{2d-i}^{BM}(X,A(d-j)).$$ The theories with and without compact supports agree for proper schemes as the canonical morphism $M(X)\to M^c(X)$ becomes the identity. Moreover, if $X$ is a smooth scheme of pure dimension $d$ and $k$ admits resolution of singularities, there is a canonical isomorphism \cite[Chapter 5, Theorem 4.3.7(3)]{VSF}: 
$$H_i(X,\Z(j))\cong H_c^{2d-i}(X,\Z(d-j)).$$

{\it Acknowledgements.} We would like to thank the referee whose valuable suggestions simplified and clarified many arguments and also improved the presentation of the paper as a whole. This paper is based on the author's master's thesis. The author would like to express his gratitude to his advisor Thomas Geisser, who suggested the topic of this paper. He would also like to thank Lars Hesselholt and Rin Sugiyama for their numerous helpful suggestions and comments in the course of writing his thesis.

\section{Borel-Moore homology}
In this preliminary section, we review some results on Borel-Moore homology groups. Let us begin with an explicit computation of all negative motivic homology groups of a finite field following \cite[Remark 2.6]{KY}.

\begin{lem}\label{rem:calculation}
Let $i$ and $j$ be negative integers. Then,
\begin{displaymath}
H_i(\Spec \F_{q},\Z(j))=\left\{\begin{array}{ll}
\F_{q^{-j}}^\times & \textrm{if $i=-1,$}\\
0 & \textrm{otherwise.}
\end{array}\right.
\end{displaymath}
\end{lem}

\begin{proof}
For $i<j\leq-1,$ $H_i(\Spec\F_{q},\Z(j))\cong CH_j(\Spec\F_{q},i-2j)=0$ for dimension reasons. For $j\leq i,$ consider the long exact sequence
$$\cdots\to H^{-i-1}(\Spec\F_{q},\Q(-j))\to H^{-i-1}(\Spec\F_{q},\Q/\Z(-j))$$
$$\to H^{-i}(\Spec\F_{q},\Z(-j))\to H^{-i}(\Spec\F_{q},\Q(-j))\to\cdots.$$
The first and the last terms vanish because 
$$H^{t}(\Spec\F_{q},\Q(-j))=CH_j(\F_{q},-t-2j)_\Q\hookrightarrow K_{-t-2j}(\F_{q})_\Q=0$$
if $-t-2j\neq0.$ The embedding follows from Bloch's Riemann-Roch Theorem~\cite[Theorem 9.1]{Bloch} and the last equality by Quillen's calculation of $K$-groups of finite fields~\cite[Theorem 8]{Quillen}. Now, because we are in the range $-i-1<-j,$ we may apply \cite[Theorem 8.4]{GLcharp} and \cite[Corollary 1.2]{GLblochkato} whose hypotheses are satisfied by the theorem of Rost and Voevodsky to obtain
\begin{eqnarray}
H_{i}(\Spec\F_{q},\Z(j)) &\cong& H^{-i}(\Spec\F_{q},\Z(-j))\nonumber\\
& \cong & H^{-i-1}(\Spec\F_{q},\Q/\Z(-j))\nonumber\\
& \cong & \bigoplus_l H^{-i-1}(\Spec\F_{q},\Q_l/\Z_l(-j))\nonumber\\
& \cong & \bigoplus_{l\neq p} H_{\acute et}^{-i-1}(\Spec\F_{q},\Q_l/\Z_l(-j)).\nonumber
\end{eqnarray} 

Hence, if $i\leq-3,$ $$H_i(\Spec \F_{q},\Z(j))\cong\bigoplus_{l\neq p} H_{\acute et}^{-i-1}(\Spec\F_{q},\Q_l/\Z_l(-j))=0$$ for $\F_{q}$ has cohomological dimension $1$.

If $i=-1,$
\begin{eqnarray}
H_{-1}(\Spec\F_{q},\Z(j))&\cong&\bigoplus_{l\neq p} H_{\acute et}^0(\Spec\F_{q},\Q_l/\Z_l(-j))\nonumber\\
&\cong&\bigoplus_{l\neq p}\Z/l^{r_l}\cong \Z/(q^{-j}-1)\cong\F_{q^{-j}}^\times,\nonumber
\end{eqnarray}
where $r_l$ is the number such that $q^{-j}-1=\prod_l l^{r_l}.$ 

Finally, for $i=-2,$  we need to show that the group $H_{\acute et}^1(\Spec\F_{q},\Q_l/\Z_l(-j))$ vanishes for an arbitrary prime $l\neq p.$ Since this is a Galois cohomology group of a finite field with torsion coefficients, there is an exact sequence
$$0\to H_{\acute et}^0(\Spec\F_{q},\Q_l/\Z_l(-j))\to \Q_l/\Z_l(-j)\buildrel id-Frob\over\longrightarrow  \Q_l/\Z_l(-j)$$
$$\to H_{\acute et}^1(\Spec\F_{q},\Q_l/\Z_l(-j))\to0.$$ Since $\Q_l/\Z_l(-j)$ is a divisible group, one can easily see that the homomorphism $id-Frob$ is either zero or surjective. As we have seen above, the first term of this exact sequence is a finite group. Thus, $id-Frob$ is not the zero map, so it must be surjective. This shows the vanishing of $H_{\acute et}^1(\Spec\F_{q},\Q_l/\Z_l(-j)).$
\end{proof}

\begin{rem}\label{rem:basefield}
Since motivic homology is defined for schemes of finite type over some base field, for Lemma~\ref{rem:calculation} to make sense, we need to specify the base field of $\Spec{\F_q}.$ However, as the proof shows, the lemma holds for any choice of the base field. More generally, see Lemma~\ref{lem:independence}.
\end{rem}

Let us further evaluate other relevant motivic invariants for later use.

\begin{lem}\label{lem:preliminary}
{\rm (i)} $H^i(\Spec\F_q,\Q(j))=0$ unless $(i,j)=(0,0).$

{\rm (ii)} If $K$ is a finitely generated field of transcendence degree 1 over $\F_q,$ then $H^i(\Spec K,\Q(j))=0$ unless $(i,j)=(0,0)$ or $(1,1).$
\end{lem}

\begin{proof}
{\rm (i)} For dimension reasons, the cohomology group in question vanishes unless $0\leq j$ and $i\leq j.$ By Bloch's Riemann-Roch Theorem \cite[Theorem 9.1]{Bloch}, there is an embedding $H^i(\Spec\F_q,\Q(j))\hookrightarrow K_{-i+2j}(\F_q)_\Q.$ Since the positive degree $K$-groups of a finite field are torsion \cite[Theorem 8]{Quillen}, this implies $H^i(\Spec\F_q,\Q(j))=0$ when $-i+2j\geq1.$ Hence, the group $H^i(\Spec\F_q,\Q(j))$ vanishes unless $0\leq j,$ $i\leq j$ and $-i+2j\leq0,$ i.e. $(i,j)=(0,0).$

{\rm (ii)} Let us first note that, for dimension reasons, the cohomology group in question vanishes unless $i\leq j$ and $j\geq0.$

Let $X$ be a smooth projective curve over $\F_q$ with function field $K.$ For a nonempty open subscheme $U$ of $X,$ the localization sequence for higher Chow groups yields an exact sequence (because $X,$ $U$ and $X\setminus U$ are all smooth)
$$H^{i-2}(X\setminus U,\Q(j-1))\longrightarrow H^i(X,\Q(j))\longrightarrow H^i(U,\Q(j))\longrightarrow H^{i-1}(X\setminus U,\Q(j-1)).$$
If $j\neq1$ or $i<j=1,$ the above two cohomology groups of $X\setminus U$ vanish by (i). Thus we obtain an isomorphism 
$$H^i(X,\Q(j))\buildrel\sim\over\longrightarrow H^i(U,\Q(j)).$$
Taking the colimit over nonempty open subschemes $U$ of $X$ and applying \cite[Lemma 3.9]{MVW}, we obtain an isomorphism
\begin{equation}
H^i(X,\Q(j))\buildrel\sim\over\longrightarrow H^i(\Spec K,\Q(j)).
\end{equation}
By Bloch's Riemann-Roch Theorem, $H^i(X,\Q(j))$ is a subgroup of $K_{-i+2j}(X)_\Q,$ and this group vanishes if $-i+2j>0$ by Harder's theorem \cite{Harder}. This means that $H^i(\Spec K,\Q(j))=0$ unless $(i,j)=(0,0)$ or $(1,1).$
\end{proof}

\begin{lem}\label{lem:rationalvanishing}
Let $X$ be a smooth curve over $\F_q.$ Then $H^i(X,\Q(j))=0$ unless $(i,j)=(0,0), (1,1)$ or $(2,1).$
\end{lem}

\begin{proof}
Let $K$ be the function field of $X.$ By the same argument to construct the isomorphism (1) in the proof of Lemma ~\ref{lem:preliminary} (ii), we obtain for an arbitrary smooth curve $X,$ a canonical isomorphism
$$H^i(X,\Q(j))\buildrel\sim\over\longrightarrow H^i(\Spec K,\Q(j))$$
if $j\neq1$ or $j=1$ but $i\neq1,2.$

Thus, by Lemma~\ref{lem:preliminary} (ii), $H^i(X,\Q(j))=0$ when $j\neq1$ unless $(i,j)=(0,0),$ and $H^i(X,\Q(j))$ also vanishes when $j=1$ and $i\neq 1,2.$ Hence, the lemma follows.
\end{proof}

The next two lemmas are special cases of \cite[Proposition 4.1, Lemma 4.2]{KY}. In that paper, a more general claim is proved under the assumption of Parshin's conjecture. The lemmas below are the part where Parshin's conjecture is not necessary. We include their proofs for the convenience of the reader.

\begin{lem}[{\cite[Lemma 4.2]{KY}}]\label{lem:colimit}
Let $X$ be an irreducible scheme of dimension $d\geq1$ over $\F_q.$ Then, for $i,j\leq-1$ with $i-j\leq1,$ 
$$\colim_U H_i^{BM}(U,\Z(j))=0,$$
where $U$ runs through the set of all nonempty open subschemes of $X.$
\end{lem}

\begin{proof}
Let $K$ denote the function field of $X.$ By definition, we have
$$\colim_U H_i^{BM}(U,\Z(j))= \colim_U CH^{d-j}(U,i-2j)\cong CH^{d-j}(\Spec K,i-2j).$$
Hence, for dimension reasons, $\colim_U H_i^{BM}(U,\Z(j))=0$ if $d-j>i-2j,$ i.e. $d>i-j.$

It remains to prove that $CH^{d-j}(\Spec K,i-2j)=0$ for $d\leq i-j,$ i.e. when $d=1.$ In this case, observe that we have the equality $i=j+1$ and the inequality $j=i-1\leq-2.$ Hence,
$$CH^{1-j}(\Spec K, i-2j)=CH^{1-j}(\Spec K,1-j)\cong K_{1-j}^M(K),$$
but the last group vanishes by Bass and Tate's calculation of Milnor $K$-groups of degree $\geq3$ of a global field \cite[II, Theorem 2.1 (3)]{BT}.  
\end{proof}

\begin{prop}[{\cite[Proposition 4.1]{KY}}]\label{thm:KYParshin}
Let $X$ be a scheme over $\F_q.$ Then, we have $H_i^{BM}(X,\Z(j))=0$ if $i\leq -2,$ $j\leq-1$ and $i-j\leq1.$
\end{prop}

\begin{proof}
We prove this by induction on the dimension of $X.$ We may suppose that $X$ is a reduced scheme because the Borel-Moore homology groups of a scheme and its reduction are the same. Thus, when $\dim X=0,$ it is enough to show the claim for $\Spec\F_{q^n}.$ This case was treated in Lemma~\ref{rem:calculation} and Remark~\ref{rem:basefield}.

Suppose, now, that the proposition is true for dimension $\leq d-1.$ Let us first prove the claim for an {\it irreducible} scheme $X$ of dimension $d.$ The localization sequence for Borel-Moore homology gives the exact sequence:
$$\colim_Y H_i^{BM}(Y,\Z(j))\longrightarrow H_i^{BM}(X,\Z(j))\longrightarrow \colim_Y H_i^{BM}(X\setminus Y,\Z(j)),$$
where $Y$ runs through the set of reduced closed subschemes of $X$ whose underlying sets are proper subsets of that of $X.$ Since the first term is zero by the induction hypothesis and the last vanishes by Lemma~\ref{lem:colimit}, we obtain $H_i^{BM}(X,\Z(j))=0.$

Now, for a general $X,$ consider the abstract blow-up
\begin{displaymath}
\xymatrix{ Z' \ar[r]^{i'} \ar[d]_{f'} & \coprod_n X_n \ar[d]^f \\
 Z \ar[r]_i & X,}
\end{displaymath}
where the $X_n$ are the irreducible components of $X$ and $Z$ is the reduced closed subscheme of $X$ where $f$ is not an isomorphism. This gives rise to an exact sequence
$$H_i^{BM}(Z,\Z(j))\oplus\bigoplus_n H_i^{BM}(X_n,\Z(j))\buildrel i_*-f_*\over\longrightarrow H_i^{BM}(X,\Z(j))\buildrel\delta\over\longrightarrow H_{i-1}^{BM}(Z',\Z(j)),$$
where $\delta$ denotes the connecting map. Hence, the vanishing of $H_i^{BM}(X,\Z(j))$ follows from the induction hypothesis and the case of irreducible schemes.
\end{proof}

\section{With $\Q$-coefficients}

Before proving Theorem~\ref{thm:MAIN} in Section 4 under the assumption of resolution of singularities, we shall prove a weaker but unconditional result without assuming any conjectures. We use de Jong's alteration \cite{deJong} in place of smooth compactification and use results of Kelly \cite{Kelly}, in particular, the existence of a localization sequence for motivic cohomology with compact supports with $\Z[1/p]$-coefficients.

\begin{thm}\label{thm:Qvanishing}
For a smooth scheme $X$ over $\F_q$ and $i,j\leq-1$ with $i-j\leq1,$ $$H_i(X,\Q(j))=0.$$
\end{thm}

We need a lemma.

\begin{lem}\label{lem:hcvanishing}
Suppose $X$ is a scheme over $\F_q$ of dimension at most $d-1.$ Then,
$$H_c^{2d-i}(X,\Q(d-j))=0$$
if $i,j\leq0$ and $i-j\leq2.$
\end{lem}

\begin{proof}
First, observe that the result is true for all $d\geq1$ if $\dim X=0$ by Lemma~\ref{lem:preliminary}~(i).  

We shall prove the lemma for a fixed $d=d_0$ by induction on $\dim X.$ In doing so, we may suppose that the lemma is true for $d\leq d_0-1$ (by induction on $d$). Suppose that the lemma is true for schemes of dimension $\leq n-1,$ and let $X$ be a scheme of dimension $n.$ There is a localization sequence~\cite[Proposition 5.5.5]{Kelly}
$$H_c^{2d_0-i}(X_{sm},\Q(d_0-j))\to H_c^{2d_0-i}(X,\Q(d_0-j))\to H_c^{2d_0-i}(X\setminus X_{sm},\Q(d_0-j)),$$
where $X_{sm}$ is the smooth locus of $X.$ The last term vanishes by the induction hypothesis. As for the first term, if we write $X_{sm}=\coprod X_i$ (where the $X_i$ are the connected components of $X$), we have 
$$H_c^{2d_0-i}(X_{sm},\Q(d_0-j)) \cong \bigoplus_i H_c^{2d_0-i}(X_i,\Q(d_0-j)).$$
So, in order to show that the first term is zero, it suffices to show that
$$H_c^{2d_0-i}(X,\Q(d_0-j))=0$$
for a smooth integral scheme $X$ of dimension $n.$ (For schemes of smaller dimensions, the vanishing statement follows from the induction hypothesis.)

Now, let $U$ be a non-empty open subscheme of $X,$ and consider the localization sequence \cite[Proposition 5.5.5]{Kelly}
$$\cdots\to H_c^{2d_0-i-1}(X\setminus U,\Q(d_0-j))\to H_c^{2d_0-i}(U,\Q(d_0-j))$$
$$\to H_c^{2d_0-i}(X,\Q(d_0-j))\to H_c^{2d_0-i}(X\setminus U,\Q(d_0-j))\to\cdots.$$
The last term vanishes by the induction hypothesis and the first term too since the lemma is known for $d=d_0-1$ and since 
$$H_c^{2d_0-i-1}(X\setminus U,\Q(d_0-j))=H_c^{2(d_0-1)-(i-1)}(X\setminus U,\Q((d_0-1)-(j-1))).$$
Thus, we see that there is an isomorphism
$$H_c^{2d_0-i}(U,\Q(d_0-j))\buildrel\sim\over\longrightarrow H_c^{2d_0-i}(X,\Q(d_0-j)).$$ This means that in order to show the claim for $X$, it is enough to prove it for some open subscheme $U$ of $X.$ 

By de Jong's theorem \cite[Theorem 4.1]{deJong}, there is an alteration $\phi: X'\longrightarrow X$ and an open immersion $X'\hookrightarrow \widehat{X'}$ into a smooth projective integral scheme $\widehat{X'}.$ There is a nonempty open subscheme $U$ of $X$ such that the projection $U':=U\times_X X'\buildrel f\over\longrightarrow U$ is finite and \'etale of degree $\delta=[k(U'):k(U)].$ The composition
$$H_c^{2d_0-i}(U,\Q(d_0-j))\buildrel f^*\over\to H_c^{2d_0-i}(U',\Q(d_0-j))\buildrel f_*\over\to H_c^{2d_0-i}(U,\Q(d_0-j))$$
is multiplication by $\delta(\neq0),$ so it is an isomorphism. In particular, $f^*$ is injective. On the other hand, $H_c^{2d_0-i}(U',\Q(d_0-j))=0$ because $U'$ is an open subscheme of a smooth projective integral scheme $\widehat{X'}$ and 
$$H_c^{2d_0-i}(\widehat{X'},\Q(d_0-j))\cong H^{2d_0-i}(\widehat{X'},\Q(d_0-j))\cong H_{2n-2d_0+i}^{BM}(\widehat{X'},\Q(n-d_0+j))=0$$
by Proposition~\ref{thm:KYParshin} (We used $i,j\leq0$ and $i-j\leq2$ here). Hence, by the injectivity of $f^*,$ we conclude that $H_c^{2d_0-i}(U,\Q(d_0-j))=0.$ The lemma is proved.
\end{proof}

\noindent{\it Proof of Theorem~\ref{thm:Qvanishing}}
\nopagebreak

We may assume that $X$ is an integral scheme. Let us write $d:=\dim X.$ If $U$ is an open subscheme of $X,$ the associated localization sequence for motivic cohomology with compact supports \cite[Proposition 5.5.5]{Kelly} gives an exact sequence
$$H_c^{2d-i-1}(X\setminus U,\Q(d-j))\to H_i(U,\Q(j))\to H_i(X,\Q(j))\to H_c^{2d-i}(X\setminus U,\Q(d-j)).$$
Here, we used \cite[Theorem 5.5.14 (3)]{Kelly}. By Lemma~\ref{lem:hcvanishing}, the first and the last terms vanish, so we have an isomorphism
$$H_i(U,\Q(j))\buildrel\sim\over\longrightarrow H_i(X,\Q(j)).$$

As before, by de Jong's theorem, there is an alteration $\phi:X'\longrightarrow X$ and a nonempty open immersion $X'\hookrightarrow \widehat{X'}$ into a smooth projective integral scheme $\widehat{X'}.$ There is an open subscheme $U$ of $X$ such that the projection $$U':=U\times_X X'\buildrel f\over\longrightarrow U$$
is a finite \'etale morphism of degree $\delta.$ The composition 
$$H_i(U,\Q(j))\buildrel f^*\over\longrightarrow H_i(U',\Q(j))\buildrel f_*\over\longrightarrow H_i(U,\Q(j))$$
is multiplication by $\delta,$ so it is an isomorphism. In particular, $f^*$ is injective. But $H_i(U',\Q(j))=0$  because $U'$ is an open subscheme of a smooth projective integral scheme $\widehat{X'},$ for which we know $H_i(\widehat{X'},\Q(j))=0$ by Theorem~\ref{thm:KYmain} and Lemma~\ref{rem:calculation}. Hence, we obtain $H_i(X,\Q(j))\cong H_i(U,\Q(j))=0.$ This proves the theorem.

\section{Proof of Theorem~\ref{thm:MAIN}}

{\it In the rest of this paper, we assume the existence of resolution of singularities in the sense of \cite[Chapter 4, Definition 3.4]{VSF}.} This assumption is needed even to deal with smooth schemes because our argument depends on the existence of smooth compactification. Alternatively, the reader may choose to assume that schemes have dimension at most 3.

Our proof of Theorem~\ref{thm:MAIN} goes as follows. We first prove the vanishing statement for $i\leq-2.$ For smooth schemes, this is done by showing that a smooth compactification induces an isomorphism of motivic homology groups of certain indices (Proposition~\ref{prop:smcpt}) and applying Kondo and Yasuda's result (Proposition~\ref{thm:KYParshin}). For a singular scheme, the statement is proved by induction on dimension using the abstract blow-up sequence associated with a desingularization of the scheme (Proposition~\ref{prop:resolution}). In order to prove the statement for $i=-1,$ we first deal with the smooth case by taking smooth compactifiation (Proposition~\ref{lem:smoothisom}). Then, combining these results, we construct the inverse to $\alpha_X: H_{-1}(X,\Z(j))\longrightarrow H_{-1}(\pi_0(X),\Z(j))$ by using the universal property of a certain pushout diagram of motivic homology groups. We show that this diagram is indeed cocartesian by means of Galois cohomology, following the method of \cite{KY}.

\begin{lem}\label{lem:4.1}
Let $X$ be a scheme over $\F_q$ of dimension at most $d-1.$ If $\dim X=0,$
$$H^{2d-i}(X,\Z(d-j))=0$$
for $i,j\leq0.$ If $\dim X\geq1,$ a desingularization $\tilde X\longrightarrow X$ of $X$ induces isomorphisms
$$H^{2d-i}(X,\Z(d-j))\buildrel\sim\over\longrightarrow H^{2d-i}(\tilde X,\Z(d-j))$$
for $i,j\leq1$ with $i-j\leq2.$
\end{lem}

\begin{proof}
If $\dim X=0,$ $X_{red}$ is a finite disjoint union of spectra of finite fields over $\F_q.$ Thus,
$$H^{2d-i}(X,\Z(d-j))\cong H^{2d-i}(X_{red},\Z(d-j))\cong H_{-2d+i}^{BM}(X_{red},\Z(-d+j))=0.$$
The first isomorphism follows because the motive $M(X)$ is isomorphic to the motive $M(X_{red}),$ and the second because $X_{red}$ is smooth. The last equality follows from Lemma~\ref{rem:calculation} because $d\geq1$ implies that $-2d+i\leq-2$ and $-d+j\leq-1.$

We prove the second assertion by induction on the dimension of $X.$ Let $Z$ be the closed subscheme of $X$ on which $\tilde X\longrightarrow X$ is not an isomorphism. Then, the abstract blow-up
\begin{displaymath}
\xymatrix{ Z' \ar[r]^{inc'} \ar[d]_{f'} & \tilde X \ar[d]^{f} \\
 Z \ar[r]_{inc} & X.}
\end{displaymath}
gives rise to a long exact sequence of motivic cohomology groups
$$\cdots\longrightarrow H^{2d-i-1}(Z',\Z(d-j))\longrightarrow H^{2d-i}(X,\Z(d-j))\buildrel (f^*,inc^*)\over\longrightarrow$$
$$H^{2d-i}(\tilde X,\Z(d-j))\oplus H^{2d-i}(Z,\Z(d-j))\buildrel inc'^*-f'^*\over\longrightarrow H^{2d-i}(Z',\Z(d-j))\longrightarrow\cdots.$$
In order to show that $f^*$ is an isomorphism, we shall prove that the three cohomology groups of $Z$ and $Z'$ vanish. 

If $\dim X=1,$ then $\dim Z=\dim Z'=0;$ so $Z_{red}$ and $Z'_{red}$ are finite disjoint unions of spectra of finite fields over $\F_q$. Hence, it suffices to observe that the first claim of the lemma implies for any finite field $\F$ over $\F_q$ the following (Note that we are in the range $d-1\geq1$):
$$H^{2d-i-1}(\Spec \F,\Z(d-j))=H^{2(d-1)-(i-1)}(\Spec \F,\Z((d-1)-(j-1)))=0,$$
$$H^{2d-i}(\Spec \F,\Z(d-j))=H^{2(d-1)-(i-2)}(\Spec \F,\Z((d-1)-(j-1)))=0.$$

Now, suppose that $\dim X\geq2$ and assume that the lemma is known for schemes of smaller dimension. We shall again prove that the cohomology groups of $Z$ and $Z'$ in the above long exact sequence vanish. Let us prove $H^{2d-i}(Z,\Z(d-j))=0.$ Since $\dim Z<\dim X,$ by the induction hypothesis, we have an isomorphism  
$$H^{2d-i}(Z,\Z(d-j))\buildrel\sim\over\to H^{2d-i}(\tilde Z,\Z(d-j)),$$
where $\tilde Z$ is a desingularization of $Z.$ Since $\tilde Z$ is a smooth scheme, every connected component $Z_r$ $(r=1,\cdots,r_0)$ of $\tilde Z$ is also smooth. By \cite[Theorem 19.1]{MVW},
\begin{eqnarray}
H^{2d-i}(\tilde Z,\Z(d-j))&\cong&\bigoplus_{r=1}^{r_0}H^{2d-i}(Z_r,\Z(d-j))\nonumber\\
&\cong&\bigoplus_{r=1}^{r_0} H^{BM}_{2\dim Z_r-2d+i}(Z_r,\Z(\dim Z_r-d+j)).\nonumber
\end{eqnarray}
The last group vanishes if $i\leq2,$ $j\leq1$ and $i-j\leq3.$ Indeed, since $\dim Z_r\leq d-2,$ when $i$ and $j$ satisfy these inequalities, we have
$$2\dim Z_r-2d+i\leq 2(d-2)-2d+i=i-4\leq-2,$$
$$\dim Z_r-d+j\leq (d-2)-d+j=j-2\leq-1,$$
and
$$(2\dim Z_r-2d+i)-(\dim Z_r-d+j)=\dim Z_r-d+i-j\leq i-j-2\leq1.$$ Hence, by Proposition~\ref{thm:KYParshin}, we obtain
$$\bigoplus_{r=1}^{r_0} H^{BM}_{2\dim Z_r-2d+i}(Z_r,\Z(\dim Z_r-d+j))=0.$$

Similarly, we can calculate $H^{2d-i}(Z',\Z(d-j))=0$ if $i\leq2,j\leq1$ and $i-j\leq3,$ and $H^{2d-i-1}(Z',\Z(d-j))=0$ if $i\leq1,$ $j\leq1$ and $i-j\leq2$. Therefore, the lemma follows.
\end{proof}

\begin{prop}\label{prop:smcpt}
Suppose $X$ is a connected smooth scheme of dimension $d$ over $\F_q.$ Then, a smooth compactification $X\hookrightarrow X'$ of $X$ induces isomorphisms
$$H_i(X,\Z(j))\buildrel\sim\over\longrightarrow H_i(X',\Z(j))$$
for all $i\leq-1$ and $j\leq0$ with $i-j\leq1.$
\end{prop}

\begin{proof}
There is nothing to prove if $d=0,$ so we deal with the case where $d\geq1.$ Since there is a localization sequence
$$H_c^{2d-i-1}(X'\setminus X,\Z(d-j))\to H_i(X,\Z(j))\to H_i(X',\Z(j))\to H_c^{2d-i}(X'\setminus X,\Z(d-j))$$
and $X'\setminus X$ is proper,
it suffices to show $H^{2d-i-1}(X'\setminus X,\Z(d-j))=0$ and $H^{2d-i}(X'\setminus X,\Z(d-j))=0.$ We prove this for the first group. The proof for the second group is identical. First, note that by the first assertion of Lemma~\ref{lem:4.1}, we may assume that the irreducible components of $X'\setminus X$ have non-zero dimension. Now, let $\widetilde{X'\setminus X}$ be a desingularization of $X'\setminus X$ and write its decomposition into connected components as $\widetilde{X'\setminus X}=\coprod_{s=1}^{s_0}X_s.$ With the second assertion of Lemma~\ref{lem:4.1}, we can calculate
\begin{eqnarray}
H^{2d-i-1}(X'\setminus X,\Z(d-j)) &\cong& H^{2d-i-1}(\widetilde{X'\setminus X},\Z(d-j))\nonumber\\
&\cong& \bigoplus_{s=1}^{s_0}H^{2d-i-1}(X_s,\Z(d-j))\nonumber\\
&\cong& \bigoplus_{s=1}^{s_0}H_{2\dim X_s-2d+i+1}^{BM}(X_s,\Z(\dim X_s-d+j)),\nonumber
\end{eqnarray}
and the last group vanishes if $i\leq-1,$ $j\leq0$ and $i-j\leq1$ by Proposition~\ref{thm:KYParshin}.
\end{proof}

We are now able to prove the first half of Theorem~\ref{thm:MAIN}.

\begin{prop}\label{prop:resolution}
Let $X$ be a scheme over $\F_q.$ Then, $H_i(X,\Z(j))=0$ if $i\leq -2,$ $j\leq-1$ and $i-j\leq1.$
\end{prop}

\begin{proof}
If $\dim X=0,$ the proposition holds by Lemma~\ref{rem:calculation}. Let us assume that $\dim X\geq1$ and prove the proposition by induction on $\dim X.$ Let $Z$ be a closed subscheme of $X$ which contains all singular points of $X$ and has dimension less than that of $X.$ The abstract blow-up
\begin{displaymath}
\xymatrix{ Z' \ar[r]^{inc'} \ar[d]_{f'} & \tilde X \ar[d]^f\\
 Z \ar[r]_{inc} & X}
\end{displaymath}
gives rise to a long exact sequence
$$H_i(Z',\Z(j))\buildrel (f'_*,inc'_*) \over\longrightarrow H_i(Z,\Z(j))\oplus H_i(\tilde X,\Z(j))\buildrel inc_*-f_*\over\longrightarrow H_i(X,\Z(j))$$
$$\longrightarrow H_{i-1}(Z',\Z(j)).$$
By the induction hypothesis, 
\begin{center}
$H_i(Z',\Z(j))=0,$ $H_i(Z,\Z(j))=0$ and $H_{i-1}(Z',\Z(j))=0.$
\end{center}
Hence,
$$H_i(X,\Z(j))\cong H_i(\tilde X,\Z(j))\cong H_i({\tilde X}',\Z(j))=0,$$
where ${\tilde X}'$ denotes a smooth compactification of $\tilde X,$ the second isomorphism follows from Proposition~\ref{prop:smcpt}, and the last group vanishes by Proposition~\ref{thm:KYParshin}.
\end{proof}

Next, we shall consider the case where $i=-1.$

\begin{lem}\label{lem:geoconn}
Let $X$ be a geometrically connected scheme over a field $k$ and $i:X\hookrightarrow X'$ be a compactification, i.e. an open immersion into a proper scheme $X'$ with dense image. Then, $X'$ is geometrically connected over $k.$
\end{lem}

\begin{proof}
Since $X'$ is connected, it is enough to show that $\pi_0(X')$ has a $k$-rational point (\cite[Chapter 10, Corollary 2.21(a)]{Liu}). Now, $i$ induces a $k$-morphism $\pi_0(X)\longrightarrow\pi_0(X').$ Since $X$ is geometrically connected over a field, $\pi_0(X)=\Spec k.$ So, this morphism defines a $k$-rational point on $\pi_0(X').$
\end{proof}

\begin{rem}\label{rem:A}
With the same notation, the above proof shows that $\pi_0(X)=\pi_0(X').$
\end{rem}

We need the independence of motivic homology from the choice of the base field.

\begin{lem}\label{lem:independence}
If $l/k$ is a finite extension of fields and $X$ is a scheme of finite type over $l,$ we have a canonical isomorphism
$$Hom_{DM_{Nis}^-(k)}(\Z(j)[i],M(X))\cong Hom_{DM_{Nis}^-(l)}(\Z(j)[i],M(X))$$
for all $i\in\Z$ and $j\in\Z_{\leq0},$ where on the left hand side, $X$ is regarded as a scheme over $k$ by the composition $X\buildrel str\over\to\Spec l\to\Spec k.$
\end{lem}

\begin{proof}
For $j<0,$ by \cite[Corollary 15.3]{MVW} and \cite[Corollary 4.10]{Voevodskybday}, there is an isomorphism 
$$Hom_{DM_{Nis}^-(l)}(\Z(j)[i],M(X))\cong H_{i-2j-1}\Big(\frac{Cor_l(\Delta^*_l,X\times_l(\mathbb A^{-j}_l-\{0\}))}{Cor_l(\Delta^*_l,X\times_l\{1\})}\Big),$$
where $Cor$ denotes the group of finite correspondences. Since $l$ is a finite extension of $k,$ if $S$ is a scheme over $l$ and $T$ is over $k,$ we have $Cor_l(T\times_k l,S)\cong Cor_k(T,S).$ Hence, the right hand side is isomorphic to 
$$H_{i-2j-1}\Big(\frac{Cor_k(\Delta^*_k,X\times_k(\mathbb A^{-j}_k-\{0\}))}{Cor_k(\Delta^*_k,X\times_k\{1\})}\Big),$$
which is, in turn, isomorphic to $Hom_{DM_{Nis}^-(k)}(\Z(j)[i],M(X)).$
\end{proof}

\begin{prop}\label{lem:smoothisom}
If $X$ is smooth over $\F_q,$ there are canonical isomorphisms
$$\phi:H_{-1}(X,\Z(j))\buildrel\sim\over\longrightarrow H_{-1}(\pi_0(X),\Z(j))$$
for $j=-1$ and $-2.$
\end{prop}

\begin{proof}
We may assume that $X$ is connected and regard it as a scheme over $\pi_0(X)$ by Lemma~\ref{lem:independence}. Now $X$ is geometrically connected as a scheme over $\pi_0(X),$ so its smooth compactification $X'$ is also smooth over $\pi_0(X)$ by Lemma~\ref{lem:geoconn}. Now, the map $\phi$ fits in the commutative diagram
\begin{displaymath}
\xymatrix{ H_{-1}(X,\Z(j)) \ar[r]^\phi \ar[d]_\sim & H_{-1}(\pi_0(X),\Z(j)) \ar@{=}[d] \\
 H_{-1}(X',\Z(j)) \ar[r]^\sim & H_{-1}(\pi_0(X'),\Z(j)),}
\end{displaymath}
where the left vertical map is an isomorphism by Proposition~\ref{prop:smcpt}, the bottom horizontal map is an isomorphism by Theorem~\ref{thm:KYmain} and the right vertical equality follows from Remark~\ref{rem:A}. Thus, $\phi$ is an isomorphism.
\end{proof}

In order to compute motivic homology of singular schemes, we shall now study how motivic homology groups behave under resolution of singularities.

\begin{lem}[{\cite[Lemma 2.7]{KY}}]\label{lem:subs}
For two finite fields $\F_{q^n}\subset\F_{q^m},$ the canonical map
$$H_{-1}(\Spec\F_{q^m},\Z(j))\longrightarrow H_{-1}(\Spec\F_{q^n},\Z(j))$$
is surjective if $j<0.$
\end{lem}

\begin{proof}
As we have seen in the proof of Lemma~\ref{rem:calculation}, the cycle class map gives an isomorphism 
$$H_{-1}(\Spec\F,\Z(j))\cong \bigoplus_{l\neq p}H_{\acute et}^0(\Spec\F,\Q_l/\Z_l(-j))$$
for $j\leq-1$ and a finite field $\F.$ Now, the cycle class map is compatible with the pushforward along a finite morphism (\cite[Lemma 3.5 (2)]{GLblochkato}), so the surjectivity follows from the corresponding statement for \'etale cohomology (\cite[Lemma 6(iii), p.269 and IV.1.7, p.283]{Soule}). 
\end{proof}

\begin{lem}\label{lem:desingisom}
Let $X$ be a scheme over $\F_q,$ $f:\tilde X\longrightarrow X$ be a desingularization, and $j=-1$ or $-2.$ Then, the map
$$f_*:H_{-1}(\tilde X,\Z(j))\longrightarrow H_{-1}(X,\Z(j))$$
is surjective.
\end{lem}

\begin{proof}
We prove this by induction on the dimension of $X.$ Let $Z$ be a closed subscheme of $X$ which contains all singularities of $X$ and has dimension less than that of $X.$ The abstract blow-up
\begin{displaymath}
\xymatrix{Z' \ar[r]^{inc'} \ar[d]_{f'} & \tilde X \ar[d]^f \\
 Z \ar[r]_{inc} & X}
\end{displaymath}
gives rise to a long exact sequence
$$H_{-1}(Z',\Z(j))\buildrel (f'_*,inc'_*)\over\longrightarrow H_{-1}(Z,\Z(j))\oplus H_{-1}(\tilde X,\Z(j))\buildrel inc_*-f_*\over\longrightarrow H_{-1}(X,\Z(j))$$
$$\buildrel\delta\over\longrightarrow H_{-2}(Z',\Z(j))=0,$$
where $\delta$ is the connecting map. The last term vanishes by Proposition~\ref{prop:resolution}. By an easy diagram chase, in order to show the surjectivity of $f_*,$ it is enough to show the surjectivity of $f'_*.$ Let us write $Z=\bigcup Z_i,$ where $Z_i$ are the irreducible components of $Z,$ and let $\tilde Z_i$ be a desingularization of $Z_i,$ and $p:\coprod \tilde Z_i\longrightarrow \cup Z_i$ be the morphism induced by the desingularizations. Note that $p$ is then a desingularization of $Z.$ For each index $i,$ choose a closed point $x_i\in\tilde Z_i.$ Let $y_i:=p(x_i)\in Z_i\subset Z$ be the image of $x_i$ under $p.$ Since $f'$ is surjective, there is a closed point $z_i\in Z'$ with $f'(z_i)=y_i$ for each $i.$ Choose some finite field extension $\F$ of $\F_q$ which contains all residue fields $k(x_i), k(y_i)$ and $k(z_i).$ The inclusions of these residue fields into $\F$ give rise to $\F$-rational points
$$x_i:\Spec\F\longrightarrow\Spec k(x_i)\longrightarrow X,$$
$$y_i:\Spec\F\longrightarrow\Spec k(y_i)\longrightarrow X,$$
and 
$$z_i:\Spec\F\longrightarrow\Spec k(z_i)\longrightarrow X,$$
which are, with an abuse of notation, denoted by the same letters $x_i,y_i$ and $z_i.$ These points give the commutative diagram
\begin{displaymath}
\xymatrix{Z' \ar[r]^-{f'} & Z=\cup Z_i \\
 \coprod_i\Spec\F \ar[u]^-{\coprod z_i} \ar[ur]^-{\coprod y_i} \ar[r]_-{\coprod x_i} & \coprod\tilde Z_i \ar[u]_-p.}
\end{displaymath}
Taking homology groups, we obtain
\begin{displaymath}
\xymatrixcolsep{2.8pc}\xymatrix{H_{-1}(Z',\Z(j)) \ar[r]^{f'_*} & H_{-1}(Z,\Z(j)) \\
 \bigoplus H_{-1}(\Spec\F,\Z(j)) \ar@{->>}[d]_-{\text{Lem.~\ref{lem:subs}}} \ar[u]^{\oplus {z_i}_*} \ar[ur]^-{\oplus {y_i}_*} \ar[r]_-{\oplus {x_i}_*} & \bigoplus H_{-1}(\tilde Z_i,\Z(j)) \ar@{->>}[u]_-{\text {induction hypothesis}} \ar[d]_-\sim^{\text{Prop~\ref{lem:smoothisom}}}\\
 \bigoplus H_{-1}(\Spec k(x_i),\Z(j)) \ar@{->>}[r]_-{\text{Lem.~\ref{lem:subs}}} & \bigoplus H_i(\pi_0(\tilde Z_i),\Z(j)).}
\end{displaymath}
Hence, $f'_*$ is surjective.
\end{proof}

The next lemma compares the motivic homology of a given scheme with the motivic homology of one of its irreducible components.

\begin{lem}\label{lem:surj}
Let $X$ be a connected scheme over $\F_q$ and $X_1$ be an irreducible component. If $j=-1$ or $-2,$ the inclusion of $X_1$ into $X$ induces a surjection
$$H_{-1}(X_1,\Z(j))\longrightarrow H_{-1}(X,\Z(j)).$$
\end{lem}

\begin{proof}
Let us write $X=X_1\cup X_2\cup\cdots\cup X_r,$ where the $X_i$ are the irreducible components of $X.$ Since the lemma is obvious for $r=1,$ we assume that $r>1$ below. The abstract blow-up
\begin{displaymath}
\xymatrix{ Z \ar[r]^\phi \ar[d] & X_r \ar[d] \\
 \cup_{i\leq r-1}X_i \ar[r]_\psi & X}
\end{displaymath}
(all the maps in the diagram are inclusions) gives an exact sequence
$$H_{-1}(Z,\Z(j))\longrightarrow H_{-1}(X_r,\Z(j))\oplus H_{-1}(\cup_{i\leq r-1}X_i,\Z(j))\longrightarrow H_{-1}(X,\Z(j))$$
$$\longrightarrow H_{-2}(Z,\Z(j))=0,$$
where the last equality comes from Propositoin~\ref{prop:resolution}. By induction on the number of irreducible components of $X,$ it suffices to prove the surjectivity of $\psi_*:H_{-1}(\cup_{i\leq r-1}X_i,\Z(j))\longrightarrow H_{-1}(X,\Z(j)),$ which, in turn, follows from the surjectivity of $\phi_*:H_{-1}(Z,\Z(j))\longrightarrow H_{-1}(X_r,\Z(j)).$

Since $X$ is connected, $Z$ is not empty. In particular, it has a closed point, say, $z\in Z.$ Choose a desingularization $\pi:\tilde X_r\longrightarrow X_r$ and let a closed point $\tilde w\in\tilde X_r$ be a preimage of $w:=\phi(z)\in X_r.$ Choose some finite field extension $\F$ of $\F_q$ containing all the residue fields $k(z),k(w)$ and $k(\tilde w),$ and regard $z,w$ and $\tilde w$ as $\F$-rational points. Now, there is a commutative diagram
\begin{displaymath}
\xymatrix{ Z \ar[r]^\phi & X_r \\
 \Spec\F \ar[u]^z \ar[r]_{\tilde w} \ar[ur]^w & \tilde X_r \ar[u]_\pi.}
\end{displaymath} 
Passing to homology groups, we obtain (noting $\mathcal O(\pi_0(\tilde X_r))\subset k(\tilde w)\subset\F$)
\begin{displaymath}
\xymatrix{ H_{-1}(Z,\Z(j)) \ar[r]^{\phi_*} & H_{-1}(X_r,\Z(j)) \\
  H_{-1}(\Spec\F,\Z(j)) \ar@{->>}[dr]_{\text{Lem.~\ref{lem:subs}}} \ar[u]^{z_*} \ar[r]^{\tilde w_*} & H_{-1}(\tilde X_r,\Z(j)) \ar@{->>}[u]_{\text{Lem~\ref{lem:desingisom}}} \ar[d]_\sim^{\text{Prop.~\ref{lem:smoothisom}}}\\
  & H_{-1}(\pi_0(\tilde X_r),\Z(j)).}
\end{displaymath}
Hence, $\phi_*$ is surjective.
\end{proof}

\noindent{\it Proof of Theorem~\ref{thm:MAIN}}

We have already proved the first half in Proposition~\ref{prop:resolution}. It remains to prove the second half, i.e. the following statement:

\vspace{5pt}

\noindent{\it Let $X$ be an arbitrary scheme over $\F_q$ and $j=-1,-2.$ Then, the canonical map $\alpha_X: H_{-1}(X,\Z(j))\longrightarrow H_{-1}(\pi_0(X),\Z(j))$ is an isomorphism.}

\vspace{5pt}

If $\dim X=0,$ then $\alpha_X$ is clearly an isomorphism (because we may assume $X$ to be a disjoint union of reduced schemes, i.e. a union of  spectra of finite fields).

We prove the theorem by induction on the dimension of $X.$ Assume that the theorem holds for schemes of dimension at most $d-1.$ We prove the assertion for a $d$-dimensional scheme $X.$ By Lemma~\ref{lem:independence}, we may assume without loss of generality that $X$ is geometrically connected and reduced. Choose a nonempty closed subscheme $Z$ of $X$ such that $X\setminus Z$ is smooth and $\dim Z<\dim X.$ First, we claim that the inclusion $Z\hookrightarrow X$ induces a surjection
$$\beta:H_{-1}(Z,\Z(j))\longrightarrow H_{-1}(X,\Z(j)).$$
Indeed, there is some irreducible component, call it $X_1,$ of $X$ such that $Z\cap X_1\neq\emptyset.$ Let $\tilde X_1$ be a desingularization of $X_1.$ Choose a closed point $x\in Z\cap X_1$ and its preimage $\tilde x\in \tilde X_1.$ Let $\F$ be a sufficiently large finite field that contains both residue fields $k(x)$ and $k(\tilde x).$ (Note that $k(\tilde x)$ contains $\pi_0(\tilde X_1).$) Then, there is a commutative diagram
\begin{displaymath}
\xymatrix{  & H_{-1}(\pi_0(\tilde X_1),\Z(j)) \\
  & H_{-1}(\tilde X_1,\Z(j)) \ar[u]^\sim_{\text{Prop.~\ref{lem:smoothisom}}} \ar@{->>}[d]^{\text{Lem.~\ref{lem:desingisom}}} \\
 H_{-1}(\Spec\F,\Z(j)) \ar[ur]^{\tilde x_*} \ar@{->>}[uur]^{\text{Lem.~\ref{lem:subs}}} \ar[d]_{x_*} \ar[r]^{x_*} & H_{-1}(X_1,\Z(j)) \ar@{->>}[d]^{\text{Lem.~\ref{lem:surj}}}\\
 H_{-1}(Z,\Z(j)) \ar[r]_\beta & H_{-1}(X,\Z(j)).}
\end{displaymath}
The commutativity of the diagram implies the surjectivity of $\beta.$

Next, consider the commutative diagram 
\begin{displaymath}
\xymatrix{ H_{-1}(Z,\Z(j)) \ar@{->>}[r]^\beta \ar[d]^{\sim\text{, ind. hypothesis}}_{\alpha_Z} & H_{-1}(X,\Z(j)) \ar[d]^{\alpha_X} \\
 H_{-1}(\pi_0(Z),\Z(j)) \ar@{->>}[r]_{\text{Lem.~\ref{lem:subs}}}^\gamma & H_{-1}(\pi_0(X),\Z(j)).}
\end{displaymath}
In order to show that $\alpha_X$ is an isomorphism, it is enough to show its injectivity for the surjectivity is obvious from the diagram. The injectivity follows once one constructs a group homomorphism
$$l:H_{-1}(\pi_0(X),\Z(j))\longrightarrow H_{-1}(X,\Z(j))$$
such that $l\circ\gamma\circ\alpha_Z=\beta$ because then, the surjectivity of $\beta$ and the equalities $l\circ\alpha_X\circ\beta=l\circ\gamma\circ\alpha_Z=\beta$ imply that $l\circ\alpha_X=id.$ (In fact, $l$ is the inverse to $\alpha_X$ because we also have $\alpha_X\circ l\circ\gamma=\alpha_X\circ\beta\circ\alpha_Z^{-1}=\gamma$ and the surjectivity of $\gamma$ implies that $\alpha_X\circ l=id.$)

The existence of such a map $l$ follows if one shows that the square in the following diagram is cocartesian.
\begin{displaymath}
\xymatrix{ H_{-1}(\pi_0(Y),\Z(j)) \ar[r] \ar[d] & H_{-1}(\pi_0(\tilde X),\Z(j)) \ar[d]^{p_*} \ar@/^/[ddr]^{p_*\circ\alpha_{\tilde X}^{-1}}\\
 H_{-1}(\pi_0(Z),\Z(j)) \ar@/_/[rrd]_{\beta\circ\alpha_Z^{-1}} \ar[r]_\gamma & H_{-1}(\pi_0(X),\Z(j)) \ar@{.>}[dr]^{\exists^! l}\\
  & & H_{-1}(X,\Z(j)). }
\end{displaymath}
Here, $p:\tilde X\longrightarrow X$ is a desingularization of $X,$ $Y:=(\tilde X\times_X Z)_{red}$ and $\gamma$ is the map induced by the canonical morphism $\pi_0(Z)\longrightarrow \pi_0(X).$ Indeed, the map $l$ defined by universality in the above diagram satisfies $l\circ\gamma\circ\alpha_Z=\beta$ by its definition. Note that $\alpha_{\tilde X}^{-1}$ makes sense because $\alpha_{\tilde X}$ is an isomorphism by Proposition~\ref{lem:smoothisom}, and so does $\alpha_Z^{-1}$ by the induction hypothesis.

Since for a zero dimensional $\F_q$-scheme $S$ and $j\leq-1$ there is an isomorphism
$$H_{-1}(S,\Z(j))\cong\bigoplus_{l\neq p} H_{\acute et}^0(S,\Q_l/\Z_l(-j))$$
that is functorial with respect to pushforward along finite morphisms \cite[Lemma 3.5 (2)]{GLblochkato} induced by the Geisser-Levine cycle map, it is enough to show that the diagram
\begin{displaymath}(*)
\xymatrix{ H_{\acute et}^0(\pi_0(Y),\Q_l/\Z_l(-j)) \ar[r] \ar[d] & H_{\acute et}^0(\pi_0(\tilde X),\Q_l/\Z_l(-j)) \ar[d]\\
 H_{\acute et}^0(\pi_0(Z),\Q_l/\Z_l(-j)) \ar[r] & H_{\acute et}^0(\pi_0(X),\Q_l/\Z_l(-j))}
\end{displaymath}
is cocartesian for all primes $l\neq p.$ (Here, the arrows are pushforward maps along finite morphisms.) 

Now, consider the diagram
\begin{displaymath}(**)
\xymatrix{ H_{\acute et}^0(\pi_0(\bar Y),\Q_l/\Z_l(-j)) \ar[r]^a \ar[d]_b & H_{\acute et}^0(\pi_0(\bar {\tilde X}),\Q_l/\Z_l(-j)) \ar[d]^c\\
 H_{\acute et}^0(\pi_0(\bar Z),\Q_l/\Z_l(-j)) \ar[r]_d & H_{\acute et}^0(\pi_0(\bar X),\Q_l/\Z_l(-j)),}
\end{displaymath}
where $\bar~$ indicates the base change to the algebraic closure $\bar\F_q,$ for example, $\bar X=X\otimes_{\F_q}\bar\F_q.$

Let us for the moment assume that the diagram (**) is cocartesian in the category of $G(\bar\F_q/\F_q)$-modules and that the module $N:=ker\{H_{\acute et}^0(\pi_0(\bar Y),\Q_l/\Z_l(-j)) \buildrel (a,b)\over\to H_{\acute et}^0(\pi_0(\bar{\tilde X}),\Q_l/\Z_l(-j))\oplus H_{\acute et}^0(\pi_0(\bar Z),\Q_l/\Z_l(-j))\}$ is divisible. We shall show that the diagram (*) is cocartesian under these assumptions. Since the diagram (**) is a pushout, there is an exact sequence of $G(\bar \F_q/\F_q)$-modules
$$0\to H_{\acute et}^0(\pi_0(\bar Y),\Q_l/\Z_l(-j))/N \buildrel (a,b)\over\to H_{\acute et}^0(\pi_0(\bar{\tilde X}),\Q_l/\Z_l(-j))\oplus H_{\acute et}^0(\pi_0(\bar Z),\Q_l/\Z_l(-j))$$ 
$$\buildrel c-d\over\to H_{\acute et}^0(\pi_0(\bar X),\Q_l/\Z_l(-j))\to0,$$
where $(a,b)$ is, of course, the quotient map induced by the map $(a,b)$ defined on $H_{\acute et}^0(\pi_0(\bar Y),\Q_l/\Z_l(-j)).$ Taking Galois cohomology of $G(\bar\F_q/\F_q)$-modules, we obtain the long exact sequence
$$0\to (H_{\acute et}^0(\pi_0(\bar Y),\Q_l/\Z_l(-j))/N)^{G(\bar\F_q/\F_q)}$$
$$ \to H_{\acute et}^0(\pi_0(\tilde X),\Q_l/\Z_l(-j))\oplus H_{\acute et}^0(\pi_0(Z),\Q_l/\Z_l(-j))\to H_{\acute et}^0(\pi_0(X),\Q_l/\Z_l(-j))$$
$$\to H^1(G(\bar\F_q/\F_q), H_{\acute et}^0(\pi_0(\bar Y),\Q_l/\Z_l(-j))/N)\to\cdots.$$
Since $j\neq0,$ the Frobenius automorphism acts nontrivially on the divisible group $H_{\acute et}^0(\pi_0(\bar Y),\Q_l/\Z_l(-j))/N,$ which is just a direct sum of copies of the divisible group $\Q_l/\Z_l(-j)/N.$ By the same reasoning as in the last part of the proof of Lemma~\ref{rem:calculation}, we conclude that 
$$H^1(G(\bar\F_q/\F_q), H_{\acute et}^0(\pi_0(\bar Y),\Q_l/\Z_l(-j))/N)=0.$$

Similarly, the short exact sequence
$$0\to N\to H_{\acute et}^0(\pi_0(\bar Y),\Q_l/\Z_l(-j))\to  H_{\acute et}^0(\pi_0(\bar Y),\Q_l/\Z_l(-j))/N\to0$$ gives rise to a long exact sequence in Galois cohomology
$$0\to N^{G(\bar\F_q/\F_q)}\to H_{\acute et}^0(\pi_0(Y),\Q_l/\Z_l(-j))$$
$$\to  (H_{\acute et}^0(\pi_0(\bar Y),\Q_l/\Z_l(-j))/N)^{G(\bar\F_q/\F_q)}\to H^1(G(\bar\F_q/\F_q), N)=0.$$
The last term vanishes because $N$ is assumed divisible and the Galois action is nontrivial if $N\neq0.$ (Since $N$ must have infinite cardinality, a trivial Galois action would imply that $N'=N^{G(\bar\F_q/\F_q)}$ is infinite, but this would contradict the fact that $N'$ is a subgroup of the finite group $H_{\acute et}^0(\pi_0(\bar Y),\Q_l/\Z_l(-j))^{G(\bar\F_q/\F_q)}.$) In particular, the map $$H_{\acute et}^0(\pi_0(Y),\Q_l/\Z_l(-j))\to  (H_{\acute et}^0(\pi_0(\bar Y),\Q_l/\Z_l(-j))/N)^{G(\bar\F_q/\F_q)}$$
is surjective.

Combining all these, we obtain an exact sequence
$$H_{\acute et}^0(\pi_0(Y),\Q_l/\Z_l(-j)) \to H_{\acute et}^0(\pi_0(\tilde X),\Q_l/\Z_l(-j))\oplus H_{\acute et}^0(\pi_0(Z),\Q_l/\Z_l(-j))$$ 
$$\to H_{\acute et}^0(\pi_0(X),\Q_l/\Z_l(-j))\to0.$$
This means that the diagram (*) is cocartesian.

It now remains to prove that the diagram (**) is cocartesian and $N$ is a divisible group. Using the Pontryagin duality
$$H_{\acute et}^0(T,\Z_l(j))\cong \mathrm{Hom}_\Z(H_{\acute et}^0(T,\Q_l/\Z_l(-j)),\Q/\Z),$$
for a zero dimensional scheme $T$ over $\bar\F_q$, obtained by taking the inverse limit over $r$ of the duality
$$H_{\acute et}^0(T,\Z/l^r(j))\cong \mathrm{Hom}_\Z(H_{\acute et}^0(T,\Z/l^r(-j)),\Q/\Z),$$
we can see that it suffices to prove that the diagram with pullback homomorphisms
\begin{displaymath}
\xymatrix{ H_{\acute et}^0(\pi_0(\bar Y),\Z_l(j))  & H_{\acute et}^0(\pi_0(\bar{\tilde X}),\Z_l(j)) \ar[l]_{a'}\\
 H_{\acute et}^0(\pi_0(\bar Z),\Z_l(j)) \ar[u]^{b'} & H_{\acute et}^0(\pi_0(\bar X),\Z_l(j)) \ar[u] \ar[l]}
\end{displaymath}
is cartesian and the cokernel of
$$H_{\acute et}^0(\pi_0(\bar{\tilde X}),\Z_l(j)) \oplus H_{\acute et}^0(\pi_0(\bar Z),\Z_l(j))\buildrel a'+b'\over\longrightarrow H_{\acute et}^0(\pi_0(\bar Y),\Z_l(j))$$
is torsion free. But, since there are canonical isomorphisms 
$$H_{\acute et}^0(\bar X,\Z_l(j))\cong \mathrm{Hom}_{Set}(\pi_0(\bar X),\Z_l)\otimes_{\Z_l}\Z_l(j)$$
of $G(\bar\F_q/\F_q)$-modules (Observe that these groups are just direct sums of $\Z_l(j)$ with one summand for each connected component of $\bar X$), it boils down to showing that the diagram
\begin{displaymath}
\xymatrix{ \pi_0(\bar Y) \ar[r]^\phi \ar[d]_\psi & \pi_0(\bar{\tilde X}) \ar[d]\\
 \pi_0(\bar Z) \ar[r] & \pi_0(\bar X),}
\end{displaymath}
where $\phi$ and $\psi$ are the canonical maps is cocartesian in the category of sets and the cokernel of the map
$$\mathrm{Hom}_{Set}(\pi_0(\bar{\tilde X}),\Z_l)\oplus\mathrm{Hom}_{Set}(\pi_0(\bar Z),\Z_l)\longrightarrow\mathrm{Hom}_{Set}(\pi_0(\bar Y),\Z_l)$$ sending $(f,g)$ to $f\circ\phi+g\circ\psi$ is torsion free. The claim about the cokernel is straightfoward. (The proof can be found in \cite[Lemma 3.3]{KY}.)

Let us prove the assertion on the square diagram. Because the map $\psi$ is surjective and $\pi_0(\bar X)$ consists of one element as we are working with a geometrically connected scheme $X$, it suffices to show that any two elements $x_1$ and $x_2$ in $\pi_0(\bar{\tilde X})$ are related by the equivalence relation generated by the relation $\sim$ on $\pi_0(\bar{\tilde X})$ defined by $s\sim s'$ if there are $t,t'\in\pi_0(\bar Y)$ such that $\phi(t)=s,$ $\phi(t')=s'$ and $\psi(t)=\psi(t').$ In order to prove this, we may assume that $x_1$ and $x_2$ in $\pi_0(\bar{\tilde X})$ correspond to irreducible components $C_1$ and $C_2$ of $\bar X$ with non-empty intersection $C_1\cap C_2.$ (If $x_1$ and $x_2$ correspond to irreducible elements $C_1$ and $C_2$ with empty intersection, choose a sequence of elements $x_1=s_1, s_2,\cdots, s_{r-1}, s_r=x_2\in\pi_0(\bar{\tilde X})$ such that their corresponding irreducible components $C_1=S_1,S_2,\cdots,S_{r-1},S_r=C_2$ of $\bar X$ have the property that $S_i\cap S_{i+1}\neq\emptyset$ for $i=1,\cdots,r-1.$ Then, apply the above case successively to $s_i$ and $s_{i+1}.$) Now, since $C_1$ and $C_2$ intersect, choose $y\in C_1\cap C_2.$ Clearly, $\bar X$ is not smooth at $y.$ Since $\bar X\setminus \bar Z$ is smooth, it follows that $y\in\bar Z.$ Choose $y_1,y_2\in\bar{\tilde X}$ lying above $y$ such that $y_1$ belongs to $C_1$ and $y_2$ to $C_2.$ By definition of $Y,$ $y_1$ and $y_2$ belong to $\bar Y$ and the connected components to which they belong are denoted by the same letters. We then have $\phi(y_1)=x_1, \phi(y_2)=x_2$ and $\psi(y_1)=y=\psi(y_2).$ This proves the theorem.


\end{document}